\documentclass[psamsfonts]{amsart}

 \makeatletter
       \def\@makefnmark{%
               \leavevmode
               \raise.9ex\hbox{\check@mathfonts
                       \fontsize\sf@size\z@\normalfont%
                               \@thefnmark}%
       }
       \makeatother
       
\usepackage{amsmath,amssymb,mathrsfs}
\usepackage[dvips,dvipdfm,usenames]{color}
\usepackage[all]{xy}
\usepackage[dvipdfm]{graphicx}
\usepackage{wrapfig}
\usepackage{tabularx}
\usepackage{supertabular}
\usepackage{setspace}
\usepackage{xypic}
\usepackage{enumerate} % option \begin{enumerate}[(a)]

\newtheorem{dfn}{Definition}[section]
\newtheorem{thm}[dfn]{Theorem}
\newtheorem{lem}[dfn]{Lemma}
\newtheorem{prp}[dfn]{Proposition}
\newtheorem{rmk}[dfn]{Remark}
\newtheorem{cor}[dfn]{Corollary}

\newtheorem{mthm}{Main\ Theorem}

\usepackage[setpagesize=false,%
dvipdfm,bookmarks=false]{hyperref}

\newcommand{\ind}{\mathrm{index}}
\newcommand{\sgn}{\mathrm{sgn}}
\newcommand{\der}{\mathrm{d}}

\title[The Roe-Higson index theorem in Riemannian surfaces]
{Toeplitz operators and the Roe-Higson type index theorem in Riemannian surfaces}
\author{Tatsuki SETO}

\address{Graduate School of Mathematics, Nagoya University, Furocho, Chikusaku, Nagoya, Japan}
\email{m11034y@math.nagoya-u.ac.jp}

\begin{document}

\begin{abstract}
We study a two dimensional analogue of the Roe-Higson index theorem for a partitioned manifold. 
We prove that Connes' pairing of some invertible element with Roe's cyclic one-cocycle 
coincides to the Fredholm index of a Toeplitz operator. 
In the proof of this paper, we use some properties of a circle and use Higson's argument. 
In the last section, 
there is a example of partitioned manifold, 
which is not a cylinder, with non-trivial pairing. 
\end{abstract}
\maketitle

\section*{Introduction}

Let $M$ be a partitioned complete Riemannian manifold, that is,  
there exist codimension zero submanifolds with boundary $M^{\pm}$ 
and a codimension one closed submanifold $N$ such that 
$M = M^{+} \cup M^{-}$ and $N = M^{+} \cap M^{-} = \partial M^{+} = \partial M^{-}$.
Let $S \to M$ be a Clifford bundle in the sense of 
\cite[Definition 3.4]{MR1670907} % Roe, John, Elliptic operators, topology and asymptotic methods. 
and let $D$ be the Dirac operator of $S$. 
Let $S_{N}$ be the restriction on $N$ of $S$. 
Then we can assume $S_{N}$ is a graded Clifford bundle over $N$. 
Let $D_{N}$ be the graded Dirac operator of $S_{N}$. 
In the above setting, 
we denote by $u_{D}$ the Cayley transform of $D$, 
that is, $u_{D} := (D-i)(D+i)^{-1}$. 
Then $u_{D}$ is a invertible element in the Roe algebra $C^{\ast}(M)$. 
In 
\cite{MR996446}, % Roe, John, Partitioning noncompact manifolds and the dual {T}oeplitz problem
Roe defined the
odd index class odd-ind$(D) \in K_{1}(C^{\ast}(M))$, which is represented by $u_{D}$, 
and he also defined the cyclic one-cocycle $\zeta$, 
which is called the Roe cocycle, on a dense subalgebra of $C^{\ast}(M)$. 
Then Connes' pairing of cyclic cohomology with $K$-theory $\langle u_{D} , \zeta \rangle$ 
is agree up to constant with 
the Fredholm index of $D_{N}^{+}$ 
\cite{MR996446}. %  Roe, John, Partitioning noncompact manifolds and the dual {T}oeplitz problem
In 
\cite{MR1113688}, % Higson, Nigel, A note on the cobordism invariance of the index
Higson gave a very simple and clear proof of this theorem. 
So we call that this theorem is the Roe-Higson index theorem. 

When $M$ is even dimensional, $\ind (D_{N}^{+})$ is always zero
in the Roe-Higson index theorem. 
So we want to get another non trivial formula. 
On the other hand, Toeplitz operators play a role of fundamental operators 
in the Atiyah-Singer index theorem on odd dimensional closed manifolds 
\cite[\S 20, 24]{MR679698}. % Baum-Douglas, {$K$} homology and index theory 
Therefore we shall prove another index theorem with Toeplitz operators on $N$. 
Namely, we shall prove Connes' pairing of the Roe cocyle with another $K_{1}$-element 
$[u_{\phi}]$ (see Proposition \ref{u_phi})
is agree with the Fredholm index of a Toeplitz operator on $N$ (Theorem \ref{mainthm}). 

\begin{mthm}
Let $M$ be a partitioned oriented Riemannian surface. 
Let $S$ be a graded spin bundle over $M$ with $\mathbb{Z}_{2}$-grading $\epsilon$ 
and let $D$ be the graded Dirac operator on $M$. 
We assume $\phi \in C^{1}(M \,;\, GL_{l}(\mathbb{C}))$ 
satisfies $\| \phi \| < \infty$, $\| grad(\phi ) \| < \infty$ and $\| \phi^{-1} \| < \infty$. 
Define 
\[
u_{\phi} := 
(D + \epsilon )^{-1}
\begin{bmatrix} \phi & 0 \\ 0 & 1 \end{bmatrix}
(D + \epsilon ). 
\]
Then the following formula holds: 
\[ \left\langle [u_{\phi}]
, \zeta \right\rangle 
	= -\frac{1}{8\pi i}\ind (T_{\phi}). \]
\end{mthm}

The main strategy of the proof of our main theorem 
is to reduce the general two dimensional case 
to the $\mathbb{R} \times S^{1}$ case by similar argument of Higson in 
\cite{MR1113688}. % Higson, Nigel, A note on the cobordism invariance of the index
To prove the $\mathbb{R} \times S^{1}$ case, we can use the family of very useful functions $\{ e^{ikx} \}_{k}$, 
which is 
orthonormal basis of $L^{2}(S^{1})$, 
all eigenvectors of Dirac operator $-i\partial / \partial x$ on $S^{1}$, 
and all representative elements of fundamental group $\pi_{1} (S^{1})$.  

The general dimensional case of our main theorem is in \cite{Setoeven}. 
It contains the $KK$-theoretic construction of $[u_{\phi}]$.

\section{Main Theorem}

\subsection{Elements of the $K_{1}$ group}

In this subsection, we define $K_{1}$ elements used in our main theorem. 

\begin{dfn}
Let $M$ be a oriented complete Riemannian manifold. 
We assume that the triple $(M^{+}, M^{-}, N)$ satisfies the following conditions: 
\begin{itemize}
\item $M^{+}$ and $M^{-}$ are two codimension zero submanifolds of $M$ with boundary, 
\item $M = M^{+} \cup M^{-}$, 
\item $N$ is a codimension one closed submanifold of $M$, 
\item $N = M^{+} \cap M^{-} = -\partial M^{+} = \partial M^{-}$.
\end{itemize}
Then we call that $(M^{+}, M^{-}, N)$ is a partition of $M$. 
Then we call that $M$ is a partitioned manifold. 
\end{dfn}

\begin{figure}[h]
 \begin{center}
  \includegraphics[width=70mm]{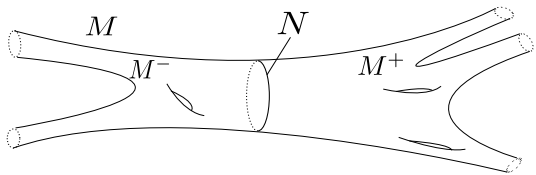}
  \caption{Partitioned manifold}
  \label{fig:integral}
 \end{center}
\end{figure}

In this paper, we assume that $M$ is a partitioned oriented two dimensional complete Riemannian manifold 
(i.e. complete Riemannian surface) 
and $(M^{+}, M^{-}, N)$ is a partition of $M$. 
Let $S$ be a graded spin bundle of $M$
\footnote{Every orientable surfaces are spin \cite[p.88]{MR1031992}. } 
with grading $\epsilon$ and a Clifford action $c$. 
% Lawson-Michelson, Spin geometry 
Let $D$ be the graded Dirac operator of $S$. 
For simplicity, we assume that $M$ is connected and 
$N$ is isometric to the unit circle $S^{1}$. 
We also assume $c(\der / \der t)c(\der / \der x) = 
\begin{bmatrix} i & 0 \\ 0 & -i \end{bmatrix}$ on $(-\epsilon , \epsilon) \times N$ 
of tubular neighborhood of $N$, where 
% $\der / \der t$ is a unit tangent vector field on $(-\epsilon , \epsilon)$, 
% $\der / \der x$ is a unit tangent vector field on $N$ and 
$\{ \der / \der t , \der / \der x \}$ 
is a positively oriented orthonormal vector fields on $(-\epsilon , \epsilon) \times N$. 

\begin{rmk}
In particular, we assume $M$ is non compact, 
then $S$ is trivial bundle: $S = M \times \mathbb{C}^{2}$. 
Especially, if $M=\mathbb{R}^{2} \cong \mathbb{C}$ with standard metric, then 
\[
D = 2
\begin{bmatrix}
0 & -\partial / \partial \bar{z} \\
\partial / \partial z & 0 
\end{bmatrix}. 
\]
\end{rmk}

Let $\mathcal{L}(L^{2}(S))$ be the set of all 
bounded operators on $L^{2}$-sections of $S$. 
Let $C^{\ast}(M)$ be the Roe algebra of $M$, that is, 
$C^{\ast}(M)$ is the completion in $\mathcal{L}(L^{2}(S))$ of 
the $\ast$-algebra of all bounded integral operators on $L^{2}(S)$ with a smooth kernel and
finite propagation 
\cite[p.191]{MR996446}. % Roe, John, Partitioning noncompact manifolds and the dual {T}oeplitz problem
We collect some well-known properties of the Roe algebra which we use. 

\begin{prp}\cite{MR1817560,MR996446}
% Higson, Roe, Analytic K-homology
% Roe, John, Partitioning noncompact manifolds and the dual {T}oeplitz problem
We assume that $M$, $S$, and $D$ are as above. 
The following holds.
\begin{enumerate}[$(1)$]
\item Let $f \in C_{0}(\mathbb{R})$ be a continuous function on $\mathbb{R}$ with 
		vanishing at infinity and let $\lambda \in \mathbb{R}$. 
		Define $D' := D + \begin{bmatrix} 0 & \lambda \\ \lambda & 0 \end{bmatrix}$, 
		then $f(D') \in C^{\ast}(M)$. 
\item Let $D^{\ast}(M)$ be the unital $C^{\ast}$-algebra generated by all 
	pseudolocal operators on $L^{2}(S)$
	\footnote{$T \in \mathcal{L}(L^{2}(S))$ is pseudolocal if $[f, T] \sim 0$ for all $f \in C_{0}(M)$
	, that is, $[f, T]$ is compact. } 
	with finite propagation. Then 
	$C^{\ast}(M)$ is a closed $\ast$-bisided ideal of $D^{\ast}(M)$. 
\item $fu \sim 0$ and $uf \sim 0$ for all $u \in C^{\ast}(M)$ and $f \in C_{0}(M)$. 
\item Let $\varpi$ be the characteristic function of $M^{+}$. 
	Then $[\varpi , u] \sim 0$ for all $u \in C^{\ast}(M)$.
\end{enumerate}
\end{prp}

By using above properties, we define a $K_{1}$-element. 

\begin{prp}
\label{u_phi}
Let $\phi \in C^{1}(M ; GL_{l}(\mathbb{C}))$ be a continuously differentiable map from $M$ 
to general linear group $GL_{l}(\mathbb{C})$. 
We assume that $\| \phi \| < \infty$, $\| grad (\phi ) \| < \infty$ and $\| \phi^{-1} \| < \infty$.  
Define
\[ u_{\phi} := (D + \epsilon )^{-1}
	\begin{bmatrix}
	\phi & 0 \\
	0 & 1 
	\end{bmatrix}
	(D + \epsilon ),
\]
then $u_{\phi} - \begin{bmatrix} 1 & 0 \\ 0 & \phi \end{bmatrix} \in M_{l}(C^{\ast}(M))$. 
\end{prp}

\begin{proof}

It suffices to show the $l=1$ case. 
Firstly, $(D + \epsilon )^{-1} \in C^{\ast}(M)$ 
and $\| (D + \epsilon)^{-1} \| \leq 2$
since 
$(D + \epsilon )^{-1} = (D^{2} + 1)^{-1}(D + \epsilon ) \in C^{\ast}(M)$. 
On the other hand,  
\begin{align*} 
u_{\phi}\sigma
&= (D + \epsilon )^{-1}
	\begin{bmatrix} \phi & 0 \\ 0 & 1 \end{bmatrix}
	\begin{bmatrix} 1 & D^{-} \\ D^{+} & -1 \end{bmatrix}\sigma \\
&= (D + \epsilon )^{-1}
	\begin{bmatrix} \phi & D^{-}\phi - D^{-}\phi + \phi D^{-} \\ D^{+} & -1 \end{bmatrix} \sigma \\
&= (D + \epsilon )^{-1}\left( (D + \epsilon ) \begin{bmatrix} 1 & 0 \\ 0 & \phi \end{bmatrix}
	+ \begin{bmatrix} \phi - 1 & [\phi , D^{-}] \\ 0 & \phi - 1 \end{bmatrix} \right) \sigma \\
&= \begin{bmatrix} 1 & 0 \\ 0 & \phi \end{bmatrix}\sigma + 
	(D + \epsilon )^{-1}\begin{bmatrix} \phi - 1 & -c(grad (\phi ))^{-} \\ 0 & \phi - 1 \end{bmatrix}\sigma ,
\end{align*}
for any $\sigma \in C^{\infty}_{c}(S)$, 
where $c(grad (\phi))^{-}$ is the negative part of the Clifford action of $grad (\phi)$. 
So 
\[
\| u_{\phi}\sigma \|_{L^{2}} \leq 3 ( \| \phi \| + \| grad (\phi ) \| + 1 )\| \sigma \|_{L^{2}}. 
\]
Thus $u_{\phi}$ can be extended uniquely as a bounded operator on $L^{2}(S)$ 
since $C^{\infty}_{c}(S)$ is dense in $L^{2}(S)$, 
and $u_{\phi} - \begin{bmatrix} 1 & 0 \\ 0 & \phi \end{bmatrix} \in C^{\ast}(M)$. 
\end{proof}

Let $C_{b}(M)$ be the set of all bounded continuous functions on $M$. 
We assume that $C^{\ast}_{b}(M) := C^{\ast}(M) + C_{b}(M)$ is a unital $C^{\ast}$-subalgebra 
of $\mathcal{L}(L^{2}(S))$. 
By this proposition, $u_{\phi}$ is invertible in $M_{l}(C^{\ast}_{b}(M))$ 
with $(u_{\phi})^{-1} = u_{\phi^{-1}}$. 
So we can consider 
$[u_{\phi}] \in K_{1}(C^{\ast}_{b}(M))$.

\subsection{Cyclic cocycle and pairing}

Let $\varpi$ be the characteristic function of 
$M^{+}$, 
and $\chi := 2\varpi -1$. 
We note that $[\chi , u]$ is a compact operator for all $u \in C^{\ast}_{b}(M)$ 
since $[\chi , f] = 0$ for all $f \in C_{b}(M)$.  
We define the Banach algebra 
\[ \mathscr{A}_{b} := \{ A \in C^{\ast}_{b}(M)  \,;\, [\chi , A] \text{ is a trace class operator} \} \]
with norm $\| A \|_{\mathscr{A}_{b}} := \| A \| + \| [\chi , A] \|_{1}$, where 
$\| \cdot \|$ is a operator norm on $L^{2}(S)$ and $\| \cdot \|_{1}$ is a trace norm. 
We define a cyclic cocycle on $\mathscr{A}_{b}$ and take the pairing of it with $K_{1}(C^{\ast}_{b}(M))$.

\begin{prp}
\cite[Proposition 1.6]{MR996446} % Roe, John, Partitioning noncompact manifolds and the dual {T}oeplitz problem
For any $A,B \in \mathscr{A}_{b}$, we define
\[ \zeta (A,B) := \frac{1}{4}\mathrm{Tr}(\chi [\chi , A][\chi , B]). \] 
Then $\zeta$ is a cyclic one-cocycle on $\mathscr{A}_{b}$. 
We call that $\zeta$ is the Roe cocycle. 
\end{prp}

To take Connes' pairing of the Roe cocycle $\zeta$ with $K_{1}(C^{\ast}_{b}(M))$, 
we use a next fact. 

\begin{prp}
\cite[p.92]{MR823176} % Connes, Alain, Noncommutative differential geometry
$\mathscr{A}_{b}$ is dense and closed under holomorphic functional calculus in $C^{\ast}_{b}(M)$.
So the inclusion $i : \mathscr{A}_{b} \to C^{\ast}_{b}(M)$ induces the
isomorphism $i_{\ast} : K_{1}(\mathscr{A}_{b}) \cong K_{1}(C^{\ast}_{b}(M))$. 
\end{prp}

\begin{proof}
Let $\mathscr{X}$ be the $\ast$-algebra of all integral operators 
on $L^{2}(S)$ with a smooth kernel and  finite propagation speed. 
Then, $\mathscr{X}_{b} := \mathscr{X} + C_{b}(M)$ is a dense subalgebra in $C^{\ast}_{b}(M)$ 
and $\mathscr{X}_{b} \subset \mathscr{A}_{b}$ 
\cite[Proposition 1.6]{MR996446}. % Roe, John, Partitioning noncompact manifolds and the dual {T}oeplitz problem 
So $\mathscr{A}_{b}$ is dense in $C^{\ast}_{b}(M)$. 

The rest of proof is in 
\cite[p.92]{MR823176}. % Connes, Alain, Noncommutative differential geometry
\end{proof}

Using this proposition, 
we can take the pairing of the Roe cocycle with $K_{1}(C^{\ast}_{b}(M))$ through the isomorphism 
$i_{\ast} : K_{1}(\mathscr{A}_{b}) \cong K_{1}(C^{\ast}_{b}(M))$ as follows: 

\begin{dfn}
\cite[p.109]{MR823176} % Connes, Alain, Noncommutative differential geometry
Define the map 
\[ \langle \cdot , \zeta \rangle : K_{1}(C^{\ast}_{b}(M)) \to \mathbb{C} \]
by $\langle [u] , \zeta \rangle := \frac{1}{8\pi i}\sum_{i,j}\zeta ((u^{-1})_{ji}, u_{ij})$, 
where we assume $[u]$ is represented by a element of $GL_{n}(\mathscr{A}_{b})$
and $u_{ij}$ is the $(i,j)$-component of $u$. 
We note that this is Connes' pairing of cyclic cohomology with $K$-theory, and
$\frac{1}{8\pi i}$ is a constant of the pairing. 
\end{dfn}

The goal of this paper is to prove that 
the result of this pairing with $[u_{\phi}]$ is the Fredholm index of a Toeplitz operator. 

\subsection{Toeplitz operators}

We review Toeplitz operators on $S^{1}$ to fix notations. 

\begin{prp}
\cite{MR0102720} % Gohberg, I. C. and Kre{\u\i}n, M. G.
% Systems of integral equations on the half-line with kernels depending on the difference of the arguments
\label{Toep}
Let $\phi \in C(S^{1} \,;\, GL_{l}(\mathbb{C}))$ be a continuous map 
from $S^{1}$ to $GL_{l}(\mathbb{C})$. 
Define $\mathcal{H} := \mathrm{Span}_{\mathbb{C}}\{ e^{ikx} \,;\, k = 0,1,2, \dots \} \subset L^{2}(S^{1})$
\footnote{$\mathcal{H}$ is called the Hardy space. }  
and let $P : L^{2}(S^{1}) \to \mathcal{H}$ be the projection. 
Then for any $f \in \mathcal{H}^{l}$, we define Toeplitz operator 
$T_{\phi} : \mathcal{H}^{l} \to \mathcal{H}^{l}$
by $T_{\phi}f := P\phi f$. Then $T_{\phi}$ is a Fredholm operator
and $\ind (T_{\phi}) = - \mathrm{deg}(\det (\phi))$, 
where $\mathrm{deg}(\det (\phi ))$ is the degree of the map $\det (\phi ) : S^{1} \to \mathbb{C}^{\times}$. 
\end{prp}

We note that the Hardy space $\mathcal{H}$ is a positive eigenspace of $-i\partial / \partial x$, 
which is a Dirac operator on $S^{1}$. See also 
\cite[p.160]{MR669904}. % Baum, Paul and Douglas, Ronald G. 
% Toeplitz operators and {P}oincar\'e duality

\subsection{Main theorem}

Using above notation, 
we state our main theorem as follows. 

\begin{thm}
\label{mainthm}
We also denote by $\phi$ the restriction on $N$ of $\phi$. 
Then the following formula holds: 
\[ \left\langle [u_{\phi}]
, \zeta \right\rangle 
	= -\frac{1}{8\pi i}\ind (T_{\phi}). \]
\end{thm}

By the index theorem of Toeplitz operators (Proposition \ref{Toep}), 
right hand side of this formula is calculated by some geometric invariant of the mapping degree. 
So this theorem is a kind of index theorem (see also next section): 

\begin{cor}
Using above notation, 
\[
\ind (\varpi u_{\phi} \varpi : \varpi (L^{2}(S))^{l} \to \varpi (L^{2}(S))^{l}) = - \mathrm{deg}(\det (\phi)) . 
\]
\end{cor}

\section{The pairing and the Fredholm index}
\label{pairing}

To prove Theorem \ref{mainthm}, we firstly describe 
$\zeta (u^{-1}, u)= \sum_{i,j}\zeta ((u^{-1})_{ji}, u_{ij})$ by the 
Fredholm index of some Fredholm operator. 

\begin{prp}
\cite[IV.1]{MR1303779} % Connes, Alain, Noncommutative geometry
For any $u \in GL_{n}(\mathscr{A}_{b})$, 
\[ \zeta (u^{-1},u) = -\ind (\varpi u \varpi : \varpi (L^{2}(S))^{n} \to \varpi (L^{2}(S))^{n}). \]
\end{prp}

\begin{proof}
Since $u \in GL_{n}(\mathscr{A}_{b})$ and
\[ \varpi - \varpi u^{-1} \varpi u \varpi = - \varpi [\varpi , u^{-1}][\varpi , u]\varpi , \]
so $\varpi - \varpi u^{-1} \varpi u \varpi$ and $\varpi - \varpi u \varpi u^{-1} \varpi$ 
are trace class operators on $\varpi (L^{2}(S))^{n}$. 
Therefore we get
\[ \ind (\varpi u \varpi : \varpi (L^{2}(S))^{n} \to \varpi (L^{2}(S))^{n}) 
= \mathrm{Tr}(\varpi - \varpi u^{-1} \varpi u \varpi )
	- \mathrm{Tr}(\varpi - \varpi u \varpi u^{-1} \varpi) \]
by
\cite[p.88]{MR823176}. % Connes, Alain, Noncommutative differential geometry
So we get 
\begin{align*}
~&\ind (\varpi u \varpi : \varpi (L^{2}(S))^{n} \to \varpi (L^{2}(S))^{n}) 
= \frac{1}{4}\mathrm{Tr}(\chi [\chi, u][\chi, u^{-1}]) \\
=& \frac{1}{4}\sum_{i,j}\mathrm{Tr}(\chi [\chi , u_{ij}][\chi , (u^{-1})_{ji}]) 
= -\zeta (u^{-1},u). 
\end{align*}
\end{proof}

By this proposition and homotopy invariance of Fredholm indices, 
we get the following formula of our pairing and the Fredholm index: 
\begin{equation}
\label{mind}
\left\langle [u_{\phi}]
, \zeta \right\rangle 
	= -\frac{1}{8\pi i}\ind (\varpi u_{\phi} \varpi : \varpi (L^{2}(S))^{l} \to \varpi (L^{2}(S))^{l}).  
\end{equation}
So we shall calculate this Fredholm index.

\section{The $\mathbb{R} \times S^{1}$ case}

In this section we shall prove the $M = \mathbb{R} \times S^{1}$ case, 
and in the next section we shall reduce the general case to $M = \mathbb{R} \times S^{1}$ case. 

In this section we assume $M = \mathbb{R} \times S^{1}$. 
$\mathbb{R} \times S^{1}$ is partitioned by $(\mathbb{R}_{+} \times S^{1}, \mathbb{R}_{-} \times S^{1}, S^{1})$, 
where $\mathbb{R}_{\pm} := \{ t \in \mathbb{R} \,;\, t \geq 0 ~(\text{resp. } t \leq 0) \}$. 
Then the Dirac operator $D$ of $S = \mathbb{R} \times S^{1} \times \mathbb{C}^{2}$ 
is given by the following formula: 
\begin{equation*}
D = 
\begin{bmatrix}
0 &  -\partial / \partial t - i\partial / \partial x \\ 
\partial / \partial t - i\partial / \partial x & 0 
\end{bmatrix},
\end{equation*}
where $(t,x) \in \mathbb{R} \times S^{1}$. 
Moreover, for any $\phi \in C^{1}(S^{1} ; GL_{l}(\mathbb{C}))$, define $\phi (t,x) := \phi (x)$ 
on $\mathbb{R} \times S^{1}$. 

\subsection{Homotopy}
\label{homotopy}

To calculate $\ind (\varpi u_{\phi} \varpi : \varpi (L^{2}(S))^{l} \to \varpi (L^{2}(S))^{l})$, 
we perturb this operator by a homotopy. 
\begin{prp}
For any $s \in [0,1]$, define 
\begin{align*} 
D_{s} &:= 
	\begin{bmatrix}
	0 & -\partial / \partial t  + s/2 - i\partial / \partial x\\
	\partial / \partial t + s/2 - i\partial / \partial x & 0
	\end{bmatrix} 
= D + 
\begin{bmatrix}
0 &  s/2 \\ 
s/2 & 0 
\end{bmatrix} \\ 
\text{and~}	u_{\phi, s} &:= (D_{s} + (1-s)\epsilon )^{-1}
	\begin{bmatrix}
	\phi & 0 \\ 0 & 1 
	\end{bmatrix}
	(D_{s} + (1-s)\epsilon ). 
\end{align*}
Then $[0,1] \ni s \mapsto u_{\phi , s} \in \mathcal{L}(L^{2}(S)^{l})$ is continuous 
and $u_{\phi , s} - \begin{bmatrix} 1 & 0 \\ 0 & \phi \end{bmatrix} \in M_{l}(C^{\ast}(M))$. 
\end{prp}

\begin{proof}
It suffices to show the $l=1$ case. 
We note that  
$\| D_{s}f \|_{L^{2}} \geq s \| f \|_{L^{2}}/2$ for any $f \in domain(D_{s}) = domain(D)$ and $s \in (0,1]$. 
Moreover $D_{s}$ is self-adjoint. Therefore the spectrum of $D_{s}$ and  $(-s/2 , s/2)$ are disjoint, 
especially $D_{1}^{-1} \in \mathcal{L}(L^{2}(S))$. 

Since $(D_{s} + (1-s)\epsilon)^{2} = D_{s}^{2} + (1-s)^{2}$, 
so $(D_{s} + (1-s)\epsilon)^{-1} \in C^{\ast}(M)$. 
Therefore $u_{\phi , s}$ is well-defined as a densely defined closed operator of $domain(u_{\phi , s}) = domain(D)$. 
By similar proof of Proposition \ref{u_phi}, 
\[ u_{\phi , s} = \begin{bmatrix} 1 & 0 \\ 0 & \phi \end{bmatrix} + 
	(D_{s} + (1-s)\epsilon )^{-1}
	\begin{bmatrix} (1-s)(\phi - 1) & i\partial \phi/ \partial x \\ 0 & (1-s)(\phi - 1) \end{bmatrix} \]
and $u_{\phi , s} - \begin{bmatrix} 1 & 0 \\ 0 & \phi \end{bmatrix} \in C^{\ast}(M)$. 

Next we show 
$\| u_{\phi , s} - u_{\phi , s'} \| \to 0$ as $s \to s'$ for all $s' \in [0,1]$. 
First, we show that $\{ \| (D_{s} + (1-s)\epsilon )^{-1} \| \}_{s \in [0,1]}$ is a bounded set. 
Set $f_{s}(x) := \left|\frac{x}{ x^{2} + (1-s)^{2} }\right|$ 
and $g_{s}(x) := \left|\frac{1}{ x^{2} + (1-s)^{2} }\right|$. 
By fundamental calculus, we can show that
\begin{equation*}
\sup_{|x| \geq s/2} \left| f_{s}(x)\right| \leq \frac{5}{2}
\text{~~~\ and~}
\sup_{|x| \geq s/2} \left| g_{s}(x)\right| \leq \frac{5}{4}. 
\end{equation*}
Therefore 
\begin{align*}
\| (D_{s} + (1-s)\epsilon )^{-1} \| 
&\leq \|(D_{s}^{2} + (1-s)^{2})^{-1}D_{s}\| + \|(1-s)(D_{s}^{2} + (1-s)^{2})^{-1}\| \\
&\leq \sup_{|x| \geq s/2} \left| f_{s}(x)\right| + \sup_{|x| \geq s/2} \left| g_{s}(x)\right| \leq 15/4 
\end{align*}
for all $s \in [0,1]$. 

On the other hand, 
\begin{align*}
u_{\phi , s} - u_{\phi , s'} 
=& \{ (D_{s} + (1-s)\epsilon )^{-1} - (D_{s'} + (1-s')\epsilon )^{-1} \}
	\begin{bmatrix}
	(1-s)(\phi - 1) & i\phi' \\
	0 & (1-s)(\phi - 1) 
	\end{bmatrix} \\
	&+ 
	(D_{s'} + (1-s')\epsilon )^{-1}
	\begin{bmatrix}
	(s'-s)(\phi - 1) & 0 \\
	0 & (s'-s)(\phi - 1)
	\end{bmatrix} 
\end{align*}
and the second term converges to $0$ in operator norm topology as $s \to s'$, 
so we should only show that 
$\| (D_{s} + (1-s)\epsilon )^{-1} - (D_{s'} + (1-s')\epsilon )^{-1} \| \to 0$ 
as $s \to s'$. But this is proved by above uniformly boundness as follows: 
\begin{align*}
& \| (D_{s} + (1-s)\epsilon )^{-1} - (D_{s'} + (1-s')\epsilon )^{-1} \| \\
=& \| (D_{s} + (1-s)\epsilon )^{-1}((s-s')\epsilon + D_{s'} - D_{s})(D_{s'} + (1-s')\epsilon )^{-1} \| \\
\leq & \frac{3}{2} |s-s'| \| (D_{s'} + (1-s')\epsilon )^{-1} \| \| (D_{s} + (1-s)\epsilon )^{-1} \| \\
\leq & 32 |s-s'| \to 0. 
\end{align*}
\end{proof}

By this proposition, we get
\[ \ind (\varpi u_{\phi} \varpi) = \ind (\varpi u_{\phi, 0} \varpi) = \ind (\varpi u_{\phi, 1} \varpi). \]
Since
\[ \varpi u_{\phi , 1} \varpi = 
	\begin{bmatrix} 
	\varpi & 0 \\
	0 & \varpi (-\partial / \partial t + 1/2 - i\partial / \partial x)^{-1}
			\phi (-\partial / \partial t + 1/2 - i\partial / \partial x) \varpi
	\end{bmatrix} \left( =: 
	\begin{bmatrix}
	\varpi & 0 \\ 0 & \mathscr{T}_{\phi} 
	\end{bmatrix} \right) ,  \]
so $\ind (\varpi u_{\phi} \varpi)$ equals to
\[
\ind \Big( \mathscr{T}_{\phi}
	:  \varpi ( L^{2}(\mathbb{R})) \otimes L^{2}(S^{1})^{l} 
	\to \varpi ( L^{2}(\mathbb{R})) \otimes L^{2}(S^{1})^{l} \Big) , 
\]
where we assume that $\varpi$ is the characteristic function of $\mathbb{R}_{+}$. 

\subsection{The Hilbert transformation}
\label{Hilbert}

Let $\mathscr{F}$ be the Fourier transformation: 
\[ \mathscr{F}[f](\xi) := \int_{\mathbb{R}}e^{-ix\xi}f(x) dx. \]
Let $H$ be the Hilbert transformation
\footnote{The coefficient $i/\pi$ of the Hilbert transformation is usually $1/\pi$. 
But we use this coefficient $i/\pi$ because of $H^{2} = 1$. }: 
\[
Hf(t) := \frac{i}{\pi}\text{p.v.}\int_{\mathbb{R}}\frac{f(y)}{t-y}dy, 
\]
where p.v. is Cauchy's principal value. 
Then we denote by $\hat{P} : L^{2}(\mathbb{R}) \to \mathscr{H}_{-}$ the projection to the
$-1$-eigenspace $\mathscr{H}_{-}$ of $H$, 
that is, $\hat{P} := \frac{1}{2}(1-H)$. 
Since $\mathscr{F}$ is a invertible operator from $\varpi (L^{2}(\mathbb{R}))$ to $\mathscr{H}_{-}$, 
so 
\begin{align*}
& \ind \Big( \mathscr{T}_{\phi}
	: \varpi ( L^{2}(\mathbb{R})) \otimes L^{2}(S^{1}) \to \varpi ( L^{2}(\mathbb{R})) \otimes L^{2}(S^{1}) \Big) \\
=& \ind \Big( \mathscr{F}\mathscr{T}_{\phi}\mathscr{F}^{-1}
	: \mathscr{H}_{-} \otimes L^{2}(S^{1}) \to \mathscr{H}_{-} \otimes L^{2}(S^{1}) \Big) .  
\end{align*}

Set 
\[ 
\hat{\mathscr{T}_{\phi}} := \mathscr{F}\mathscr{T}_{\phi}\mathscr{F}^{-1} 
= \hat{P}(-it+ 1/2 - i\partial / \partial x )^{-1} 
\phi (-it + 1/2 - i\partial / \partial x) \hat{P}^{\ast}. 
\]
To calculate the Fredholm index of $\hat{\mathscr{T}_{\phi}}$, 
we use a basis of $L^{2}(\mathbb{R})$ made by eigenvectors of Hilbert transformation. 

\begin{prp}
\cite[Theorem 1]{MR1277773} 
% Weideman, J. A. C., Computing the {H}ilbert transform on the real line
Set $\rho_{n} \in L^{2}(\mathbb{R})$ by
\[ \rho_{n}(t) := \frac{(t-i)^{n}}{(t+i)^{n+1}}. \]
Then $\{ \rho_{n}/\sqrt{\pi} \}$ is a orthonormal basis of $L^{2}(\mathbb{R})$ 
and 
\begin{equation*}
H\rho_{n} = 
\begin{cases}
\rho_{n} & \text{if~} n < 0 \\
-\rho_{n} & \text{if~} n \geq 0
\end{cases}.
\end{equation*}
\end{prp}

By using this basis, we get $\mathscr{H}_{-} = \mathrm{Span}_{\mathbb{C}}\{ \rho_{n} \,;\, n \geq 0 \}$ and 
we can calculate following Fredholm indices, which is used in the next subsection. 

\begin{lem}
\label{WHind}
For any $\alpha , \beta \neq 0$, 
$\hat{P} \dfrac{t+i\beta}{t+i\alpha}\hat{P}^{\ast} \in \mathcal{L}(\mathscr{H}_{-})$ 
is a Fredholm operator and 
\begin{equation*}
\ind \left(\hat{P} \frac{t+i\beta}{t+i\alpha}\hat{P}^{\ast}\right) =
\begin{cases}
0 & \text{if~} \alpha \beta > 0 \\
-1 & \text{if~} \alpha > 0 , \beta < 0 \\
1 & \text{if~} \alpha < 0 , \beta > 0
\end{cases}.
\end{equation*}
\end{lem}

\begin{proof}

Our proof of Fredholmness is similar to 
\cite[p.99]{MR771117}. 
% Booss, B. and Bleecker, D. D., Topology and analysis
Let $c : \mathbb{R} \to S^{1} (\subset \mathbb{C})$ be the Cayley transformation defined by
$c(t) := (t-i)(t+i)^{-1}$. 
Define $\Phi(g)(t) := (t+i)^{-1}g(c(t))$ for any $g \in L^{2}(S^{1})$. 
Then $\Phi : L^{2}(S^{1}) \to L^{2}(\mathbb{R})$
is a invertible bounded linear operator with $\| \Phi \| = 1/\sqrt{2}$ and 
$\Phi^{-1}(f)(z) = (c^{-1}(z) + i)f(c^{-1}(z))$ 
for all $f \in L^{2}(\mathbb{R})$ and $z \in S^{1} \setminus \{ 1 \}$. 

Since $\Phi (e^{inx}) = \rho_{n}$, 
$\hat{P}f\hat{P}^{\ast}$ is a Fredholm operator on $\mathscr{H}_{-}$
for any $f \in C^{\infty}(\mathbb{R} ; \mathbb{C}^{\times})$ with 
$\lim_{t \to \infty}f(t) = \lim_{t \to -\infty}f(t) \in \mathbb{C}^{\times}$. 
Now, 
\[ \left| \frac{t+i\beta}{t+i\alpha} \right|^{2} 
	= \frac{t^{2} +  \beta^{2} }{t^{2} + \alpha^{2}} > 0  \]
and $\lim_{t \to \pm \infty}\frac{t+i\beta}{t+i\alpha} = 1$. 
Therefore $\hat{P} \frac{t+i\beta}{t+i\alpha}\hat{P}^{\ast}$ is a Fredholm operator. 

We calculate $\ind (\hat{P} \frac{t+i\beta}{t+i\alpha}\hat{P}^{\ast})$. 
Define $\sgn (\alpha) := \begin{cases} 1 & \text{if~} \alpha > 0 \\ -1 & \text{if~} \alpha < 0 \end{cases}$. 
Then we define a homotopy of Fredholm operators from $\hat{P} \frac{t+i\beta}{t+i\alpha}\hat{P}^{\ast}$ to
$\hat{P} \frac{t+i\sgn (\beta )}{t+i\sgn(\alpha )}\hat{P}^{\ast}$ by 
\[ \hat{P} \frac{t+i(s\beta + (1-s)\sgn (\beta))}{t+i(s\alpha + (1-s)\sgn (\alpha))}\hat{P}^{\ast} \] 
for $s \in [0,1]$. Therefore  
\[ \ind \left(\hat{P} \frac{t+i\beta}{t+i\alpha}\hat{P}^{\ast}\right) 
	= \ind \left(\hat{P} \frac{t+i\sgn (\beta )}{t+i\sgn(\alpha )}\hat{P}^{\ast}\right) 
	= \begin{cases}
		0 & \text{if~} \alpha \beta > 0 \\
		-1 & \text{if~} \alpha > 0 , \beta < 0 \\
		1 & \text{if~} \alpha < 0 , \beta > 0
	  \end{cases}
\]
by $\mathscr{H}_{-} = \mathrm{Span}_{\mathbb{C}}\{ \rho_{n} \,;\, n \geq 0 \}$. 
\end{proof}

\subsection{The special case}

Set $\phi_{k}(x) = e^{ikx}$ on $S^{1}$ for $k \in \mathbb{Z}$. 
In this subsection, we calculate 
\[ \ind (\hat{\mathscr{T}_{\phi_{k}}} 
: \mathscr{H}_{-} \otimes L^{2}(S^{1}) \to \mathscr{H}_{-} \otimes L^{2}(S^{1}) ) \]
of the special case. 

\begin{prp}
\label{phik}
$
\ind (\varpi u_{\phi_{k}} \varpi) = \ind (\hat{\mathscr{T}_{\phi_{k}}}) = -k = \ind (T_{\phi_{k}})$. 
\end{prp}

\begin{proof}
The first equality is proved in subsection \ref{homotopy} and \ref{Hilbert}, 
and the last equality is well known. 
So we should only show the second equality. 
Let $E_{\lambda} := \mathbb{C}\{ e^{i\lambda x} \}$ be 
the $\lambda$-eigenspace of $-i\partial / \partial x$. 
On $\mathscr{H}_{-} \otimes E_{\lambda}$, operator $\hat{\mathscr{T}}_{\phi_{k}}$ 
acts as 
\[ 
\hat{P}(-it + 1/2 + \lambda + k)^{-1} 
(-it + 1/2 + \lambda) \hat{P}^{\ast} \otimes \phi_{k} 
\]
and $\hat{\mathscr{T}_{\phi_{k}}}(\mathscr{H}_{-} \otimes E_{\lambda})$ 
is contained in $\mathscr{H}_{-} \otimes E_{\lambda + k}$.
Therefore 
\begin{align*}
& \ind (\hat{\mathscr{T}_{\phi_{k}}}) \\
%=& \sum_{\lambda = -\infty}^{\infty} \ind (\phi_{k} \otimes \hat{P}(\lambda + k + 1/2 - it)^{-1} 
%	(k + 1/2 - it) \hat{P} \in E_{\lambda + k} \otimes \mathscr{H}_{-}) \\
=& \sum_{\lambda = -\infty}^{\infty} \ind {\left( \hat{P}
	\frac{t + i(\lambda + 1/2)}{t + i(\lambda + k + 1/2)} \hat{P}^{\ast} \otimes \phi_{k} :
	\mathscr{H}_{-} \otimes E_{\lambda} \to \mathscr{H}_{-} \otimes E_{\lambda + k}\right)} \\
=&  -k
\end{align*}
by Lemma \ref{WHind}.
\end{proof}

\subsection{The general case}

Let $\phi \in C^{1} (S^{1} ; GL_{l}(\mathbb{C}))$ be a general continuously differentiable map. 
We reduce $\phi$ case to $\phi_{k}$ case. 

Since the Gram-Schmidt orthogonalization $GL_{l}(\mathbb{C}) \to U(l)$ 
is a homotopy equivalence map, so 
this map induces the isomorphism on fundamental groups 
$\pi_{1}(GL_{l}(\mathbb{C})) \cong \pi_{1}(U(l))$. 
By the homotopy long exact sequence, this inclusion $i : C(S^{1} ; S^{1}) \to C(S^{1} ; U(l))$ 
of $i(f) := \begin{bmatrix} f & 0 \\ 0 & 1_{l-1} \end{bmatrix}$
induces the isomorphism on fundamental groups 
$i_{\ast}: \pi_{1}(S^{1}) \to \pi_{1}(U(l))$ 
\cite[Example 4.55]{MR1867354}. 
Moreover $\pi_{1}(S^{1}) \cong \mathbb{Z}$ is represented by $\phi_{k}$ for all $k \in \mathbb{Z}$. 
Therefore $\phi$ is homotopic to $\begin{bmatrix} \phi_{k} & 0 \\ 0 & 1_{l-1} \end{bmatrix}$ 
in $C(S^{1}\,;\,GL_{l}(\mathbb{C}))$
for some $k \in \mathbb{Z}$. We denote $\psi_{s}$ by this homotopy. 
Moreover, since $C^{1}(S^{1})$ is dense and closed under holomorphic functional calculus in $C(S^{1})$, 
so we can take this homotopy $\psi_{s}$ in $C^{1}(S^{1} ; GL_{l}(\mathbb{C}))$. 

\begin{prp}
\label{lasthomotopy}
$u_{\psi_{s}}$ is a continuous path in $\mathcal{L}(L^{2}(S)^{l})$. 
\end{prp}

\begin{proof}

By proof of Proposition \ref{u_phi}, 
\[ u_{\psi_{s}} = \begin{bmatrix} 1 & 0 \\ 0 & \psi_{s} \end{bmatrix} + 
	(D + \epsilon )^{-1}\begin{bmatrix} \psi_{s} - 1 & i\psi_{s}' \\ 0 & \psi_{s} - 1 \end{bmatrix}. \]
Therefore 
\[ \| u_{\psi_{s}} - u_{\psi_{s'}} \| \leq 3 \| \psi_{s} - \psi_{s'} \|_{C^{1}}
	\to 0 \]
as $s \to s'$ since $\psi_{s}$ is a path in $C^{1}(S^{1} ; GL_{l}(\mathbb{C}))$. 
\end{proof}

\begin{proof}[Proof of Theorem \ref{mainthm} of $M = \mathbb{R} \times S^{1}$ case]

By Proposition \ref{lasthomotopy}, 
\[ 
\ind (\varpi u_{\phi} \varpi) 
	= \ind \left(\varpi \begin{bmatrix} u_{\phi_{k}} & 0 \\ 0 & 1_{l-1} \end{bmatrix} \varpi \right)
	= \ind (\varpi u_{\phi_{k}} \varpi)
\] 
for some $k \in \mathbb{Z}$. 
By Proposition \ref{phik}, 
$\ind (\varpi u_{\phi_{k}} \varpi) = \ind (T_{\phi_{k}})$. 
Moreover, since $\phi$ is homotopic to $\begin{bmatrix} \phi_{k} & 0 \\ 0 & 1_{l-1} \end{bmatrix}$, 
so $\ind (T_{\phi_{k}}) = \ind (T_{\phi})$. 
Therefore $\ind (\varpi u_{\phi} \varpi) = \ind (T_{\phi})$. 
We complete a proof of Theorem \ref{mainthm} by using (\ref{mind}) in section \ref{pairing}. 
\end{proof}

\section{The general two-manifold case}

In this section we reduce the general two dimensional manifold case to the $\mathbb{R} \times S^{1}$ case. 
Our argument is similar to Higson's argument in 
\cite{MR1113688}. % Higson, Nigel, A note on the cobordism invariance of the index
Firstly, we shall show cobordism invariance of the pairing. 

\begin{lem}
\cite[Lemma 1.4]{MR1113688} % Higson, Nigel, A note on the cobordism invariance of the index
\label{cobor}
Let $(M^{+}, M^{-}, N)$ and $(M^{+}{}', M^{-}{}' , N')$ be two partitions of $M$. 
Then we assume these two partitions are cobordant, that is, 
the symmetric differences $M^{\pm} \triangle M^{\mp}{}'$ are compact. 
Let $\varpi$ and $\varpi'$ be the characteristic function of $M^{+}$ and $M^{+}{}'$, respectively. 
We assume $\phi \in C^{1}(M \,;\, GL_{l}(\mathbb{C}))$ satisfies 
$\| \phi\| < \infty$, $\| grad (\phi) \| < \infty$ and $\| \phi^{-1} \| < \infty$. 
Then $\ind (\varpi u_{\phi} \varpi) = \ind (\varpi' u_{\phi} \varpi')$. 
\end{lem}

\begin{proof}
It suffices to show the $l=1$ case. 
Since $[\phi , \varpi ] = 0$ and $[u_{\phi} , \varpi ] \sim 0$, so 
\begin{align*}
 & \ind (\varpi u_{\phi} \varpi : \varpi (L^{2}(S)) \to \varpi (L^{2}(S))) \\
=& \ind \left((1- \varpi ) 
	\begin{bmatrix} 1 & 0 \\ 0 & \phi \end{bmatrix} + \varpi u_{\phi} \varpi : L^{2}(S) \to L^{2}(S) \right) \\
=& \ind \left((1- \varpi ) 
	\begin{bmatrix} 1 & 0 \\ 0 & \phi \end{bmatrix} + \varpi u_{\phi} : L^{2}(S) \to L^{2}(S) \right) \\
=& \ind \left(
	\begin{bmatrix} 1 & 0 \\ 0 & \phi \end{bmatrix} + \varpi v_{\phi} : L^{2}(S) \to L^{2}(S) \right), 
\end{align*}
where we assume $v_{\phi} = u_{\phi} - \begin{bmatrix} 1 & 0 \\ 0 & \phi \end{bmatrix} \in C^{\ast}(M)$. 
Therefore we should only show 
$\varpi v_{\phi} \sim \varpi' v_{\phi}$. 
Now, since $M^{\pm} \triangle M^{\mp}{}'$ are compact, 
there exists $f \in C_{0}(M)$ such that $\varpi - \varpi' = (\varpi - \varpi')f$. 
So $\varpi v_{\phi} - \varpi' v_{\phi} = (\varpi - \varpi')fv_{\phi} \sim 0$. 
\end{proof}

Secondly, we shall reduce the general manifold case to the $\mathbb{R} \times N$ case
used by Higson's Lemma. 

\begin{lem}
\cite[Lemma 3.1]{MR1113688} % Higson, Nigel, A note on the cobordism invariance of the index
\label{Higson}
Let $M_{1}$ and $M_{2}$ be two partitioned oriented two dimensional complete Riemannian manifolds and
let $S_{j}$ be a Hermitian vector bundle over $M_{j}$.
Let $\varpi_{j}$ be the characteristic function of $M_{j}^{+}$. 
We assume that there exists an isometry $\gamma : M_{2}^{+} \to M_{1}^{+}$ which lifts
an isomorphism $\gamma^{\ast} : S_{1}|_{M_{1}^{+}} \to S_{2}|_{M_{2}^{+}}$. 
We denote again by $\gamma^{\ast} : \varpi_{1}(L^{2}(S_{1})) \to \varpi_{2}(L^{2}(S_{2}))$ 
the Hilbert space isometry defined by $\gamma$. 
Then we assume $u_{j} \in GL_{l}(C^{\ast}(M_{j}))$ and $f_{j} \in GL_{l}(C_{b}(M_{j}))$ 
satisfy $v_{j} := u_{j} - f_{j} \in M_{l}(C^{\ast}(M))$ and 
$\gamma^{\ast}u_{1}\varpi_{1} \sim \varpi_{2}u_{2}\gamma^{\ast}$. 
Then $\ind (\varpi_{1}u_{1}\varpi_{1}) = \ind (\varpi_{2}u_{2}\varpi_{2})$. 

Similarly, if there exists an isometry $\gamma : M_{2}^{-} \to M_{1}^{-}$ which lifts
an isomorphism $\gamma^{\ast} : S_{1}|_{M_{1}^{-}} \to S_{2}|_{M_{2}^{-}}$ and 
$\gamma^{\ast}u_{1}\varpi_{1} \sim \varpi_{2}u_{2}\gamma^{\ast}$, 
then $\ind (\varpi_{1}u_{1}\varpi_{1}) = \ind (\varpi_{2}u_{2}\varpi_{2})$.
\end{lem}

\begin{proof}
It suffices to show $l=1$ case. 
Let $v : (1-\varpi_{1})(L^{2}(S_{1})) \to (1-\varpi_{2})(L^{2}(S_{2}))$ 
be any invertible operator. 
Then $V := \gamma^{\ast}\varpi_{1} + v(1-\varpi_{1}) : L^{2}(S_{1}) \to L^{2}(S_{2})$ is also invertible operator. 
Hence 
\begin{align*}
& \; V((1-\varpi_{1}) + \varpi_{1}u_{1}\varpi_{1}) - ((1-\varpi_{2}) + \varpi_{2}u_{2}\varpi_{2})V \\
%=& \; -V\varpi_{1} + \varpi_{2}V + V\varpi_{1}u_{1}\varpi_{1}
%	- \varpi_{2}u_{2}\varpi_{2}V \\ 
=& \; -\gamma^{\ast}\varpi_{1} + \varpi_{2}\gamma^{\ast} + \gamma^{\ast}\varpi_{1}u_{1}\varpi_{1}
	- \varpi_{2}u_{2}\varpi_{2}\gamma^{\ast} \\ 
\sim & \; \gamma^{\ast}u_{1}\varpi_{1} - \varpi_{2}u_{2}\gamma^{\ast} \sim 0. 
\end{align*}
Therefore $\ind (\varpi_{1}u_{1}\varpi_{1}) = \ind (\varpi_{2}u_{2}\varpi_{2})$ 
since $V$ is an invertible operator and 
$\ind (\varpi_{j}u_{j}\varpi_{j}) = \ind ((1-\varpi_{j}) + \varpi_{j}u_{j}\varpi_{j})$ 
for $j=1,2$. 
\end{proof}

In our case, we use this Lemma as follows: 

\begin{cor}
\label{ourH}
Let $M_{j}$ be a Riemannian surface and let $S_{j}$ be a
graded spin bundle over $M_{j}$ 
with grading $\epsilon_{j}$. 
We assume that there exists an isometry $\gamma : M_{2}^{+} \to M_{1}^{+}$ 
which defines the Hilbert space isometry 
$\gamma^{\ast} : \varpi_{1}(L^{2}(S_{1})) \to \varpi_{2}(L^{2}(S_{2}))$ 
as in Lemma \ref{Higson} 
and 
satisfies $D_{2} \gamma^{\ast} \sim \gamma^{\ast}D_{1}$ and 
$\epsilon_{2} \gamma^{\ast} \sim \gamma^{\ast}\epsilon_{1}$ on $\varpi_{1}(L^{2}(S_{1}))$. 
Let $\phi_{j} \in C^{1}(M_{j} \,;\, GL_{l}(\mathbb{C}))$ satisfies 
$\|\phi_{j} \| < \infty$, $\| grad (\phi_{j}) \| < \infty$ and $\| \phi_{j}^{-1} \| < \infty$ 
as in Proposition \ref{u_phi} and 
$\phi_{1}$ and $\phi_{2}$ satisfy $\phi_{1}(\gamma (x)) = \phi_{2}(x)$ for all $x \in M_{2}^{+}$. 
Then $\ind (\varpi_{1}u_{\phi_{1}}\varpi_{1}) = \ind (\varpi_{2}u_{\phi_{2}}\varpi_{2})$. 
\end{cor}

\begin{proof}
We should only show $\gamma^{\ast}u_{\phi_{1}}\varpi_{1} \sim \varpi_{2} u_{\phi_{2}} \gamma^{\ast}$. 
Let $\varphi_{1}$ be a smooth function on $M_{1}$ such that 
$\mathrm{Supp} (\varphi_{1}) \subset M_{1}^{+}$ and there exists 
a compact set $K_{1} \subset M_{1}$ such that $\varphi_{1} = \varpi_{1}$ on $M_{1} \setminus K_{1}$. 
Define $\varphi_{2}(x) := \varphi_{1}(\gamma (x))$ for all $x \in M_{2}^{+}$ 
and $\varphi_{2} = 0$ on $M_{2}^{-}$. 
Then $\varphi_{2}$ is a smooth function on $M_{2}$ such that 
$\mathrm{Supp} (\varphi_{2}) \subset M_{2}^{+}$ and there exists 
a compact set $K_{2} \subset M_{2}$ such that $\varphi_{2} = \varpi_{2}$ on $M_{2} \setminus K_{2}$. 
Set 
$v_{\phi_{j}} = u_{\phi_{j}} - \begin{bmatrix} 1 & 0 \\ 0 & \phi_{j} \end{bmatrix}$. 
Then $\gamma^{\ast}v_{\phi_{1}}\varpi_{1} \sim \gamma^{\ast}v_{\phi_{1}}\varphi_{1}$ and 
$\varpi_{2}v_{\phi_{2}}\gamma^{\ast} \sim \varphi_{2}v_{\phi_{2}}\gamma^{\ast}$. 
So if $\gamma^{\ast}v_{\phi_{1}}\varphi_{1} \sim \varphi_{2}v_{\phi_{2}}\gamma^{\ast}$, then 
\begin{equation*}
\gamma^{\ast}u_{\phi_{1}}\varpi_{1}
\sim
\gamma^{\ast}v_{\phi_{1}}\varphi_{1} + \gamma^{\ast}\begin{bmatrix} 1 & 0 \\ 0 & \phi_{1} \end{bmatrix}\varpi_{1}
\sim
\varphi_{2}v_{\phi_{2}}\gamma^{\ast} + \varpi_{2}\begin{bmatrix} 1 & 0 \\ 0 & \phi_{2} \end{bmatrix}\gamma^{\ast}
\sim
\varpi_{2} u_{\phi_{2}} \gamma^{\ast}.
\end{equation*}
So we shall show $\gamma^{\ast}v_{\phi_{1}}\varphi_{1} \sim \varphi_{2}v_{\phi_{2}}\gamma^{\ast}$. 
In fact, 
\allowdisplaybreaks
\begin{align*}
 & \; \gamma^{\ast}v_{\phi_{1}}\varphi_{1} - \varphi_{2}v_{\phi_{2}}\gamma^{\ast} \\
=& \; \gamma^{\ast}(D_{1} + \epsilon_{1} )^{-1}
	\begin{bmatrix} \phi_{1} - 1 & -c(grad (\phi_{1} ))^{-} \\ 0 & \phi_{1} - 1 \end{bmatrix}\varphi_{1}
	- 
	\varphi_{2}(D_{2} + \epsilon_{2} )^{-1}
	\begin{bmatrix} \phi_{2} - 1 & -c(grad (\phi_{2} ))^{-} \\ 0 & \phi_{2} - 1 \end{bmatrix}\gamma^{\ast} \\
=& \; \{\gamma^{\ast}(D_{1} + \epsilon_{1} )^{-1}\varphi_{1} - \varphi_{2}(D_{2} + \epsilon_{2} )^{-1}\gamma^{\ast}\}
	\begin{bmatrix} \phi_{1} - 1 & -c(grad (\phi_{1} ))^{-} \\ 0 & \phi_{1} - 1 \end{bmatrix} \\
\sim & \; \{\gamma^{\ast}\varphi_{1}(D_{1} + \epsilon_{1} )^{-1} 
		- (D_{2} + \epsilon_{2} )^{-1}\gamma^{\ast}\varphi_{1}\}
	\begin{bmatrix} \phi_{1} - 1 & -c(grad (\phi_{1} ))^{-} \\ 0 & \phi_{1} - 1 \end{bmatrix} \\
= & \; (D_{2} + \epsilon_{2})^{-1}
	\{(D_{2} + \epsilon_{2})\gamma^{\ast}\varphi_{1} 
			- \gamma^{\ast}\varphi_{1}(D_{1} + \epsilon_{1})\}(D_{1} + \epsilon_{1})^{-1}
	\begin{bmatrix} \phi_{1} - 1 & -c(grad (\phi_{1} ))^{-} \\ 0 & \phi_{1} - 1 \end{bmatrix} \\
\sim & \; (D_{2} + \epsilon_{2})^{-1}\gamma^{\ast}[D_{1}, \varphi_{1}](D_{1} + \epsilon_{1})^{-1}
	\begin{bmatrix} \phi_{1} - 1 & -c(grad (\phi_{1} ))^{-} \\ 0 & \phi_{1} - 1 \end{bmatrix} \\
\sim & \; 0 
\end{align*}
since $grad (\varphi_{1})$ has a compact support and $[D_{1}, \varphi_{1}] = c(grad (\varphi_{1}))$. 
Thus we get $\gamma^{\ast}u_{\phi_{1}}\varpi_{1} \sim \varpi_{2} u_{\phi_{2}} \gamma^{\ast}$. 
Therefore $\ind (\varpi_{1}u_{\phi_{1}}\varpi_{1}) = \ind (\varpi_{2}u_{\phi_{2}}\varpi_{2})$ 
by Lemma \ref{Higson}. 
\end{proof}

\begin{proof}[Proof of Theorem \ref{mainthm}]
Firstly, let $a \in C^{\infty}([-1,1];[-1,1])$ satisfies 
\[ a(t) = 
\begin{cases}
-1 & \text{if~} -1 \leq t \leq -3/4 \\
0 & \text{if~} -2/4 \leq t \leq 2/4 \\
1 & \text{if~} 3/4 \leq t \leq 1
\end{cases}. 
\]
Let $(-4\delta , 4\delta ) \times N$ be a tubular neighborhood of $N$ in $M$ satisfies 
\[
\sup_{(t,x),(s,y) \in [-3\delta , 3\delta ] \times N} |\phi (t,x) - \phi (s,y)| < \| \phi^{-1} \|^{-1}. 
\]
Define $\psi (t,x) := \phi (4\delta a(t) , x)$ on $(-4\delta , 4\delta ) \times N$ and 
$\psi = \phi$ on $M \setminus (-4\delta , 4\delta ) \times N$. 
Then $\psi \in C^{1}(M \,;\, GL_{l}(\mathbb{C}))$ and $\| \psi - \phi \| < \| \phi^{-1} \|^{-1}$. 
Thus 
$[0,1] \ni t \mapsto \psi_{t} := t\psi + (1-t)\phi$  
satisfies $\psi_{t} \in C^{1}(M \,;\, GL_{l}(\mathbb{C}))$, 
$\| \psi_{t} \| < \infty$, $\| grad (\psi_{t}) \| < \infty$ and 
$\| \psi_{t}^{-1} \| < \infty$, 
$\| \psi_{t} - \psi_{t'} \| \to 0$ and $\| grad (\psi_{t}) - grad (\psi_{t'}) \| \to 0$ 
as $t \to t' \in [0,1]$. 
Therefore we should only show the case of which $\phi$ satisfies
$\phi (t,x) = \phi (0,x)$ on tubular neighborhood $(-2\delta , 2\delta) \times N$. 
By Lemma \ref{cobor} we may change a partition of $M$ 
to $(M^{+} \cup ([-\delta , 0] \times N), M^{-}\setminus ((-\delta, 0] \times N), \{ -\delta\} \times N)$ 
without changing $\ind (\varpi u_{\phi} \varpi)$. 
Then by Corollary \ref{ourH} we may change $M^{+} \cup ([-\delta , 0] \times N)$ to 
$[-\delta , \infty ) \times N$ without changing $\ind (\varpi u_{\phi} \varpi)$ 
where $\phi$ is changed to $\phi (t,x) = \phi (x)$ for $(t,x) \in (-\delta, 0] \times N$. 
We denote by 
$M' := ([-\delta , \infty ) \times N) \cup (M^{-}\setminus ((-\delta, 0] \times N))$ 
this manifold. 
Then $M'$ is partitioned by 
$([-\delta , \infty ) \times N, M^{-}\setminus ((-\delta, 0] \times N), \{ -\delta \} \times N)$. 
By using Lemma \ref{cobor} and Corollary \ref{ourH} again, 
we may change $M'$ to $\mathbb{R} \times N$ without changing $\ind (\varpi u_{\phi} \varpi)$ 
with similar argument as above where $\phi (t,x) = \phi (x)$ for $(t,x) \in \mathbb{R} \times N$. 
Now we have changed $M$ to $\mathbb{R} \times N = \mathbb{R} \times S^{1}$. 

\end{proof}

\section{Example}

In this section, we see a example of a partitioned manifold with which
the pairing $\langle [u_{\phi}] , \zeta \rangle$ is not zero 
but it is not $\mathbb{R} \times S^{1}$. 

Firstly, we define partitioned two-manifold. 
Let $\Sigma_{2}$ be a closed Riemannian surface of genus two and 
let $C$ and $C'$ be two submanifolds of $\Sigma_{2}$ which define 
generators of $H_{1}(\Sigma_{2} ; \mathbb{Z})$ as Figure \ref{fig:genus_2}. 
Moreover, we cut $\Sigma_{2}$ along $C$ and $C'$ and 
embed to $\mathbb{R}^{3}$ like Figure \ref{fig:Salpha}. 

\begin{figure}[h]
 \begin{minipage}{0.45\hsize}
  \begin{center}
   \includegraphics[width=50mm]{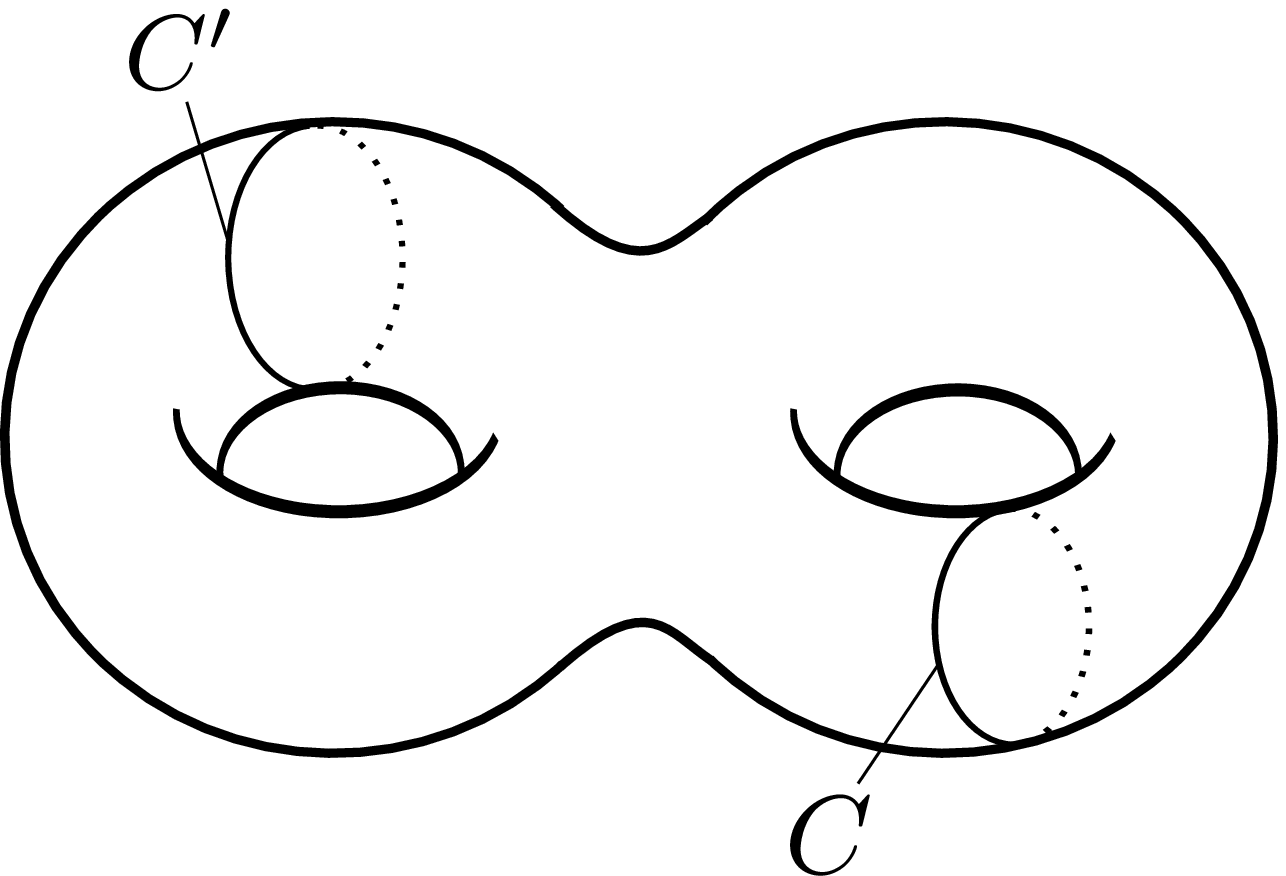}
  \end{center}
  \caption{~}
  \label{fig:genus_2}
 \end{minipage}
 \begin{minipage}{0.45\hsize}
  \begin{center}
   \includegraphics[width=65mm]{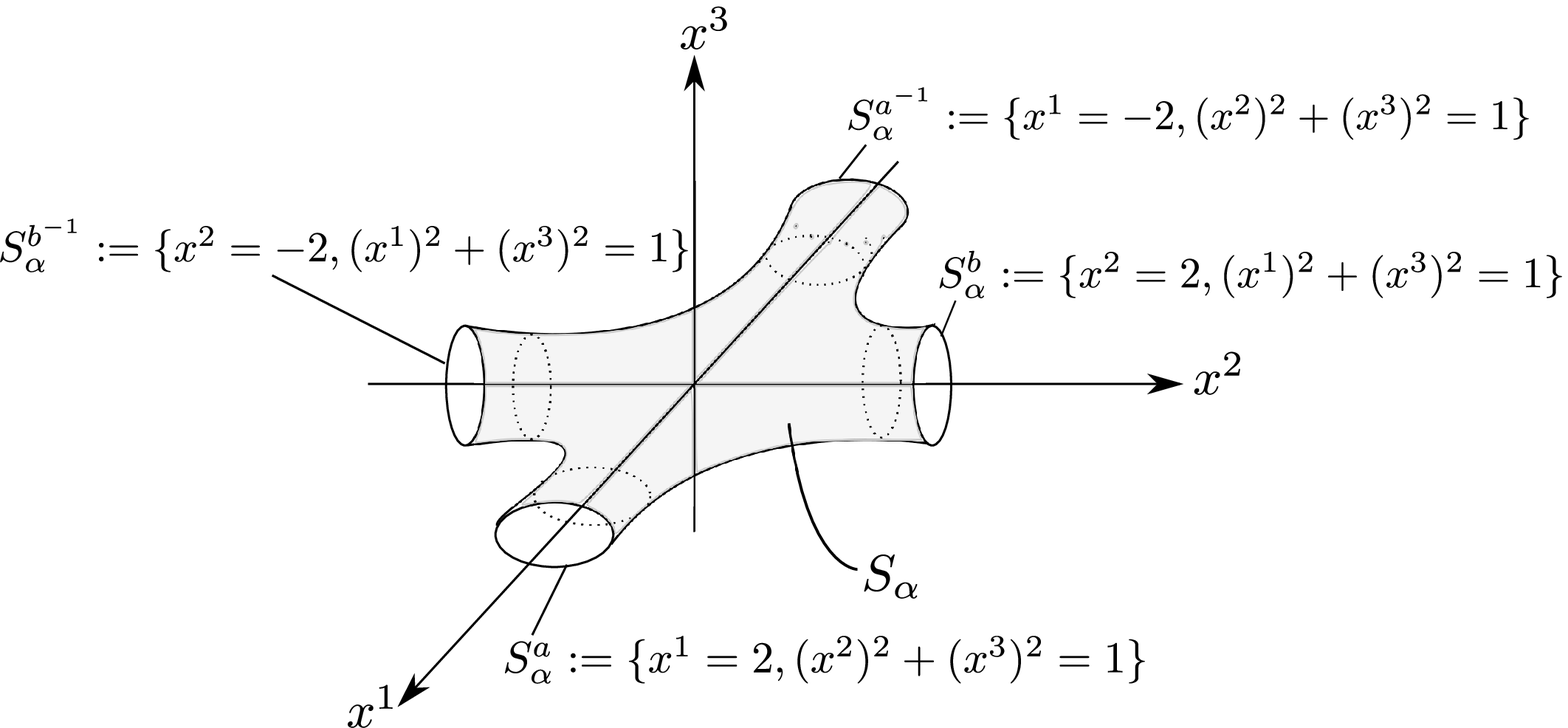}
  \end{center}
  \caption{~}
  \label{fig:Salpha}
 \end{minipage}
\end{figure}

Let $F_{2} := \langle a,b \rangle$ be the free group with two generators. 
For all $\alpha \in F_{2}$, we consider such surface $S_{\alpha}$. 
Then we assume $S_{\alpha}$ is a oriented smooth manifold with Riemannian metric induced by $\mathbb{R}^{3}$. 
Moreover we assume 
$T_{\alpha}^{g}(1/2)$ is a collar neighborhood of $S_{\alpha}^{g}$ in $S_{\alpha}$ 
for $g \in \{ a,a^{-1},b,b^{-1} \}$. 
%Moreover we assume 
%\[
%\{ x^{1} \in \mathbb{R}, (x^{2})^{2} + (x^{3})^{2} = 1 \} \cap S_{\alpha} 
%	\supset T_{\alpha}^{a}(1/2) \cup T_{\alpha}^{a^{-1}}(1/2) 
%\]
%and 
%\[
%\{ x^{2} \in \mathbb{R}, (x^{1})^{2} + (x^{3})^{2} = 1 \} \cap S_{\alpha} 
%	\supset T_{\alpha}^{b}(1/2) \cup T_{\alpha}^{b^{-1}}(1/2), 
%\]
Where for all $\delta > 0$, we define
\[
T_{\alpha}^{a}(\delta) := \{ 2-\delta < x^{1} \leq 2, (x^{2})^{2} + (x^{3})^{2} = 1 \}, 
\]
\[
T_{\alpha}^{a^{-1}}(\delta) := \{ -2 \leq x^{1} < -2 + \delta, (x^{2})^{2} + (x^{3})^{2} = 1 \}, 
\]
\[
T_{\alpha}^{b}(\delta) := \{ 2-\delta < x^{2} \leq 2, (x^{1})^{2} + (x^{3})^{2} = 1 \} \text{~, and} 
\]
\[
T_{\alpha}^{b^{-1}}(\delta) := \{ -2 \leq x^{2} < -2 + \delta, (x^{1})^{2} + (x^{3})^{2} = 1 \}. 
\]

Let $\phi_{\alpha}^{g} : \overline{T_{\alpha}^{g}(1/4)} \to [0,1/4] \times S^{1}$ be 
the orientation preserving isometry defined by identification of 
$\overline{T_{\alpha}^{g}(1/4)}$ with $[0,1/4] \times S^{1}$. 
Define $M := \bigsqcup_{\alpha}S_{\alpha}/\sim$, where 
$x \sim y$ if 
(i) $x \in \overline{T_{\alpha}^{g}(1/4)}$ and $y \in \overline{T_{\beta}^{h}(1/4)}$, 
(ii) $\alpha g = \beta h$ and  
(iii) $\phi_{\alpha}^{g}(x) = \phi_{\beta}^{h}(y)$. 
Then $M$ is a oriented complete Riemannian manifold. 
Moreover, $F_{2}$ acts on $M$ freely. 
Let $\pi : M \to M/F_{2} = \Sigma_{2}$ be the quatient map of this action. 
We note that this manifold $M$ forms like ``the boundary of a fat picture'' of Cayley graph of $F_{2}$. 

Let $N \subset \pi^{-1}(C)$ be a connected component of $\pi^{-1}(C)$, 
where $C := \pi (S_{\alpha}^{a})$. 
Then $M$ is separated two components by $N$. 
So we can define $M^{+}$ and $M^{-}$ which satisfy $N = \partial M^{-}$. 
Therefore $M$ is a partitioned manifold. 

On the other hand, there exists continuously differentiable map 
$\varphi : \Sigma_{2} \to GL_{l}(\mathbb{C})$ such that 
$\mathrm{deg}(\det (\varphi |_{C})) \neq 0$ 
since $[C] \neq 0$. 
For example, we choose $\varphi : S^{1} \to GL_{l}(\mathbb{C})$ such that $\mathrm{deg}(\det (\varphi)) \neq 0$, 
and we extend on $T^{2} = S^{1} \times S^{1}$ trivially. 
Then we can define such $\varphi$ on $\Sigma_{2}$ through $\Sigma_{2} = T^{2} \# T^{2}$. 
Define $\phi := \varphi \circ \pi$, then $\phi$ satisfies 
assumptions of Theorem \ref{mainthm}. 

In above setting, $\deg (\det (\phi |_{N})) = \deg (\det (\varphi |_{C}))$. 
Therefore $\langle [u_{\phi}] , \zeta \rangle \neq 0$. 

\vspace*{1\baselineskip}

\textit{Acknowledgments} - The author would like to thank 
his advisor Professor Hitoshi Moriyoshi for helpful conversations.

\bibliographystyle{jplain}
\bibliography{referrences_S1}

\end{document}